\documentclass[11pt,reqno]{amsart}

\usepackage{amsthm,amsfonts,amsmath,amssymb,bookmark,appendix}
\usepackage{amsmath,amsfonts,amsthm,amssymb,amsxtra}
\usepackage{amsxtra, amssymb, mathrsfs}
\usepackage[usenames,dvipsnames]{color}

\newtheorem{thm}{Theorem}[section]
\newtheorem{prop}[thm]{Proposition}

\newtheorem{lemma}[thm]{Lemma}
\newtheorem{corollary}[thm]{Corollary}
\newtheorem{assumption}[thm]{Assumption}

\theoremstyle{definition}

\newtheorem{example}[thm]{Example}

\theoremstyle{remark}
\newtheorem{remark}[thm]{\bf Remark}

\numberwithin{equation}{section}

\makeatletter
\renewcommand*\env@cases[1][1.6]{%
  \let\@ifnextchar\new@ifnextchar
  \left\lbrace
  \def\arraystretch{#1}%
  \array{@{}l@{\quad}l@{}}%
}
\makeatother

\newcommand{\norm}[1]{\mbox{$\left\| #1 \right\|$}}
\newcommand{\sprod}[2]{\mbox{$\left\langle {\,#1},{\, #2} \right\rangle$}}

\newcommand{\AM}{\mathrm{A}}

\newcommand{\eps}{\varepsilon}
\newcommand{\F}{\mathcal{F}}

\newcommand{\N}{\mathbb{N}}

\newcommand{\PP}{\mathbb{P}}
\newcommand{\R}{\mathbb{R}}
\newcommand{\RR}{\mathcal{R}}

\newcommand{\df}{\ \mathrm{d}}
\newcommand{\menge}[2]{\left\{ \, {#1}  \ \big{|} \ {#2} \, \right\}}

\renewcommand{\exp}[1]{\mathrm{e}^{#1}}

\renewcommand{\det}[1]{\mathrm{det}\left({#1} \right)}

\newcommand{\trace}[1]{\mathrm{Tr}{\left(#1\right)}}
\newcommand{\fourier}{ \mathcal{F}_{3}}

\newcommand{\lam}{\lambda_{j}(\Omega)}

\newcommand{\magnet}{\mathrm{\textbf{A}}}

\begin{document}

\title{Melas-type bounds for the Heisenberg Laplacian on bounded domains}


\author{Hynek Kova\v r\'{\i}k}
\address{Hynek Kova\v r\'{\i}k, DICATAM, Sezione di Matematica, Universit\`a degli studi di Brescia,
Via Branze, 38 - 25123  Brescia, Italy}
\email{hynek.kovarik@unibs.it}

\author{Bartosch Ruszkowski}
\address{Bartosch Ruszkowski, Institute of Analysis, Dynamics
and Modeling, Universit\"at Stuttgart, 
PF 80 11 40, D-70569  Stuttgart, Germany}
\email{Bartosch.Ruszkowski@mathematik.uni-stuttgart.de}

\author{Timo Weidl}
\address{Timo Weidl, Institute of Analysis, Dynamics
and Modeling, Universit\"at Stuttgart, 
PF 80 11 40, D-70569  Stuttgart, Germany}
\email{weidl@mathematik.uni-stuttgart.de}

\maketitle


\begin{abstract}
We study Riesz means of the eigenvalues of
the Heisenberg Laplacian with Dirichlet boundary conditions on bounded domains in $\R^3$. 
We obtain an inequality with a sharp leading term and an additional lower order term, improving the result of Hanson and Laptev, \cite{hanson}. 
\end{abstract}

\section{\bf Introduction} \label{sec-intro}

Let $\Omega \subset \R^3$ be an open bounded domain.
In this paper we consider the Heisenberg Laplacian on $L^2(\Omega)$ with Dirichlet boundary condition formally given by
\begin{align*}
\mathrm{A}(\Omega)\ := \ -X_{1}^{2}-X^{2}_{2}\, ,
\end{align*}
where
\begin{equation} \label{X12}
X_{1} \ :=  \ \partial_{x_{1}}+\frac{x_2}{2}\, \partial_{x_{3}}, \quad X_{2} \ := \ \partial_{x_{2}}-\frac{x_1}{2}\, \partial_{x_{3}}\, .
\end{equation}
More precisely, $\mathrm{A}(\Omega)$ is the unique self-adjoint operator associated with the closure of the quadratic form 
\begin{align}\label{einseins}
\mathrm{ {a}}[u]& \ :=  \ \int_{\Omega} \left( \left |X_{1}\, u(x)\right|^{2}+\left|X_{2}\, u(x) \right|^{2} \right) \df x,
\end{align}
initially defined for $u \in C_{0}^{\infty}(\Omega)$. Note that 
\begin{align*}
 [X_2,X_1]= \partial_{x_3} =: X_{3}\, .
\end{align*}
We recall that the vector fields $X_1, X_2, X_3$ form a basis of the Lie algebra of left-invariant vector fields on the first Heisenberg group $\mathbb{H}$  given by $\R^3$ and equipped with the group law
\begin{equation} \label{group-law}
(x_1,x_2,x_3)\boxplus (y_1,y_2,y_3) \ := \ \Big{ (x_1+y_1,x_2+y_2, x_3+y_3- \frac{1}{2}\, (x_1 y_2-x_2 y_1)\Big )}.
\end{equation}
The sub-elliptic  estimate proved in \cite{follandsub} shows that 
\begin{equation} \label{eq-folland}
\|u\|^2_{H^{1/2}(\Omega)} \, \leq \, c\, \left(a[u]+ \|u\|_{L^2(\Omega)}^2\right), \qquad u\in C_0^\infty(\Omega)
\end{equation}
holds for some $c>0$. Hence the domain of the closure of $a[\cdot]$ is continuously imbedded in $H_0^{1/2}(\Omega)$. Since the imbedding $H_0^{1/2}(\Omega) \to L^2(\Omega)$ is compact, by standard Sobolev imbedding theorems, it follows that the spectrum of $\mathrm{A}(\Omega)$ is purely discrete. 
We denote by $\{\lambda_k\}_{k\in \N}$ the non-decreasing sequence of the eigenvalues of $\mathrm{A}(\Omega)$ and by  $\{v_j \}_{k\in \N}$ the associated sequence of normalized eigenfunctions;
\begin{equation} \label{eq-eigen}
\mathrm{A}(\Omega) \, v_j \ = \ \lambda_{j} \,  v_{j}, \qquad \norm{v_{j}}_{L^{2}(\Omega)}=1.
\end{equation}
Recently Hanson and Laptev proved, see \cite[Thm.~2.1]{hanson}, that 
\begin{equation}  \label{eq-hl}
\trace{\, \mathrm{A}(\Omega)-\lambda}_{-}  = \sum_{k\in \N} \left(\lambda-\lambda_k \right)_+  \ \leq  \ \frac{|\Omega|}{96}\,  \lambda^{3} \qquad \forall\ \lambda >0.
\end{equation}
Here  the eigenvalues $\lambda_k$ are repeated according to their multiplicities and $|\Omega|$ denotes the Euclidean volume of $\Omega$.  Moreover, it is also shown in \cite{hanson} that the constant $\frac{1}{96}$ on the right hand-side of \eqref{eq-hl} is sharp.  Indeed, this follows from the asymptotic equation
\begin{equation}  \label{eq-hl-asymp}
\lim_{\lambda\to\infty} \lambda^{-3}\  \trace{\, \mathrm{A}(\Omega)-\lambda}_{-}  = \frac{|\Omega|}{96}\, ,
\end{equation}
see \cite[Cor.~2.8]{hanson}. 

\smallskip

\noindent The aim of this paper is to improve \eqref{eq-hl}  by adding to its right hand-side a negative term of a lower order in $\lambda$. In other words, we are going to show that 
\begin{equation}  \label{to-show}
\trace{\, \mathrm{A}(\Omega) -\lambda}_{-}    \ \leq  \ \frac{|\Omega|}{96}\,  \lambda^{3} - \mathcal{C}(\Omega)\, \lambda^\alpha,
\end{equation}
where $\mathcal{C}(\Omega)$ is a positive constant which depends only on $\Omega$ and $\alpha\in (0,3)$. In our main result, see Theorem \ref{dreieins}, we will prove inequality \eqref{to-show} with $\alpha =2$ and give an explicit expression for the constant $\mathcal{C}(\Omega)$.

This is in the spirit of Melas-type improvements of the well-known Berezin-Li-Yau inequality 
\begin{equation} \label{bly-eq}
\trace {\, -\Delta_\Omega-\lambda}_{-} \ \leq \ \frac{|\Omega|\,}{(4\pi)^{\frac d2}\, \Gamma\left(2+\frac d2\right)} \  \lambda^{1+\frac d2}\, , \qquad \Omega \subset \R^d,
\end{equation}
where $-\Delta_\Omega$ denotes the Dirichlet Laplacian on $\Omega$, see \cite{berezin,li} and \cite{melas,kovarik,kvw,yolcu,yolcu2}. In particular, our main result improves inequality \eqref{eq-hl} in a similar way in which \cite{kovarik} improves inequality \eqref{bly-eq}. 

However, the method that we employ in the present paper is different from the one used in \cite{kovarik} since it does not rely on a Hardy inequality involving the distance to the boundary. In fact, as far as we know an analog of such an inequality for the Heisenberg-Laplacian with explicit constants is not known.
Instead we exploit the properties of the Carnot-Carath\'{e}odory metric which is connected to the Heisenberg-Laplacian in a natural way, see sections \ref{sec-prelim} and \ref{erstesection} for details. 

The paper is organized as follows. The main result is announced in section \ref{sec-result}. In section \ref{erstesection}, and in particular in Theorem \ref{mainnn},  we present some auxiliary results concerning the properties of balls with respect to the  Carnot-Carath\'{e}odory metric. The proof of the main result is given in section \ref{sec-proof}. In the closing section \ref{sec-improvement} we establish an improvement of Theorem \ref{dreieins}, which reveals a better order of $\lambda$ in the additional term. However for this result we need the additional Assumption that the Hardy inequality respectively the Carnot-Carath\'{e}odory metric must be valid.

\section{\bf Main results}\label{sec-result}

\subsection{Preliminaries} \label{sec-prelim} For a fixed point $x\in \mathbb{H}$ we denote its Euclidean norm by $\norm{x}_{e}$. 
Now we introduce the \textit{\bf Carnot-Carath\'{e}odory} metric. We call a Lipschitz curve $c: [a,b]\subset \mathbb{R} \to \mathbb{H}$ 
\textit{horizontal} if $c(t):=(x_1(t),x_2(t), x_3(t))$  fulfills the following differential equation
\begin{align}\label{horiz}
x_3^{\prime}(t)= \frac{1}{2}\left(x_2(t)\, x_1^\prime(t)-x_1(t)\, x_2^\prime(t) \right)  .
\end{align}
Horizontal curves always exist because the Heisenberg group is a step two Carnot group and we can apply the Chow's theorem, se e.g.~\cite{montgomery}. For a given pair $x,y \in \mathbb{H}$ we introduce the family of curves 
\begin{equation}
\F(x,y) := \left\{ c: [a,b] \to \mathbb{H}\, :\, c \ {\rm is\, horizontal \, and \, connects}\  x\  {\rm with} \ y\right\}.   
\end{equation}
Furthermore, we set 
\begin{align}
 l_{\mathbb{H}}(c):= \int_{a}^{b}\sqrt{x_1^{\prime}(t)^{2}+x_2^\prime(t)^{2} }\ \df t.
\end{align}
Given $x,y\in \mathbb{H}$, the Carnot-Carath\'{e}odory metric (C-C metric in the sequel) is then defined as follows; 
\begin{equation} \label{CC-def}
d(x,y) : = \inf_{c \in \F(x,y)}  l_{\mathbb{H}}(c)
\end{equation}
For a more detailed introduction to this metric we refer to \cite{capogna}, \cite{montgomery}, \cite{calin}. Let
$$
B_r(0) = \left\{ x \in \mathbb{H}\, : \, d(x,0) < r \right\}
$$
be the ball with radius $r>0$, with respect to the C-C metric, centered at the origin.
Let us introduce the distance from a fixed point $x\in \Omega$ to the boundary of $\Omega$ with respect to the C-C metric, which will be denoted by
\begin{align}
d(x):= \inf\limits_{y\in \partial \Omega}d(x,y).
\end{align}
When needed, we extend the function $d(\cdot)$ on $\mathbb{H}$;  for points lying in $x\in \Omega^{c}$ we set $d(x)=0$. 
In addition we introduce the in-radius of $\Omega$, which is defined by
\begin{align} \label{in-rad-cc}
\mathrm{R}(\Omega):= \sup\limits_{x\in \Omega} d(x)\, .
\end{align}

\subsection{Main result}
With the above notation at hand we can state our main result. 

\begin{thm}\label{dreieins}
Let $\Omega\subset \mathbb{H}$ be a bounded domain. Then 
\begin{align}\label{hauptgleichung}
\trace{\mathrm{A}(\Omega) -\lambda}_{-} \ \leq  \ \max \left\{0,\  \frac{|\Omega|}{96} \ \lambda^{3}- \frac{|B_1(0)|^2\, \mathrm{R}(\Omega)^6}{3 \cdot 2^9\, |\Omega|}\ \lambda^{2} \right\} 
\end{align}
holds true for all $\lambda>0$. 
\end{thm}

\begin{remark}
Equation \eqref{eq-hl-asymp} implies that 
\begin{equation} \label{asymp-hl}
\trace{\mathrm{A}(\Omega) -\lambda}_{-} =  \frac{|\Omega|}{96} \ \lambda^{3} + o(\lambda^3) \qquad \lambda\to\infty.
\end{equation}
So far the order of the remainder term in \eqref{asymp-hl} is not known.
\end{remark}

\noindent The upper bound \eqref{eq-hl} is equivalent, by means of the Legendre transform, to the Li-Yau type lower bound 
\begin{equation} \label{hl-li-yau}
\sum_{j=1}^{n} \lambda_{j}(\Omega) \ \geq \ \frac{8\sqrt{2}}{3}|\Omega|^{-\frac 12} \ n^{\frac 32} \qquad n\in \N,
\end{equation}
see \cite[Cor.~2.10]{hanson}.
In the same way Theorem \ref{dreieins} implies an improvement of \eqref{hl-li-yau}:

\begin{corollary} 
For any $n\in\N$ it holds
\begin{align}
 \sum_{j=1}^{n} \lambda_{j}(\Omega) \ \geq \ \frac{8\sqrt{2}}{3}|\Omega|^{-\frac 12} \ n^{\frac 32}+ \frac{1}{48} \frac{ |B_1(0)|^2\, \mathrm{R}(\Omega)^6}{|\Omega|^{2}}\ n\, .
\end{align}
\end{corollary}

\begin{proof}
Let us recall that if $f,g:\R\to\R$ are two convex non-negative functions, then the implication
\begin{equation} \label{legendre}
f(x) \leq g(x), \ x\geq 0 \quad \Leftrightarrow  \quad g^{\star}(p) \leq f^{\star}(p), \ p\geq 0
\end{equation}
holds true, where $f^{\star}$ and $g^{\star}$ are Legendre transforms of $f$ and $g$ defined by
$$
f^{\star}(p) := \sup_{x\geq0} (px-f(x)), \qquad  g^{\star}(p) := \sup_{x\geq0} (px-g(x))\, .
$$
The claim now follows by applying \eqref{legendre} to \eqref{hauptgleichung} with $f(\lambda)= \trace{A(\Omega)-\lambda}_{-}$ and 
$$
g(\lambda) =  \max \left\{0,\  \frac{|\Omega|}{96} \ \lambda^{3}- \frac{|B_{\mathrm{R}(\Omega)}(0)|^{2}}{{3 \cdot 2^9\, \mathrm{R}(\Omega)}^{2}\, |\Omega|}\ \lambda^{2} \right\} \, .
$$
\end{proof}


\section{\bf Auxiliary results}
\label{erstesection}
\noindent The goal of this section is to prove a sharp lower bound on the Euclidean volume of the set
\begin{align}
 \Omega^{\beta}:= \menge{x\in \Omega}{d(x) < \beta},
\end{align}
for a given $\beta\in (0, R(\Omega))$. We start by stating some properties of the C-C metric which be needed later. 

\smallskip

\noindent The arc joining geodesics starting from the origin 
were computed in  \cite{montii} and \cite{marenich}. The parametrization of these arcs is given by
\begin{align}\label{geod}
\gamma_{k,\theta}(t):= 
\begin{cases}
x_1(t,k,\theta) =\displaystyle \frac{\cos(\theta)-\cos(kt+\theta)}{k},\\[0,3cm]
x_2(t,k,\theta)=\displaystyle \frac{\sin(kt+\theta)-\sin(\theta)}{k},\\[0,3cm]
\displaystyle x_3(t,k,\theta)=\frac{kt-\sin(kt)}{2k^{2}}\, ,
\end{cases}
\end{align}
where $t\in[0, \frac{2\pi}{{|k|}}]$ , $\theta \in [0,2\pi)$ and $k \in \mathbb{R}\setminus \{0\}$. This means that for the given point $\gamma_{k,\theta}(t)\in \mathbb{H}$  holds $d(\gamma_{k,\theta}(t),0)=t$.  We extend this formula to the case $k=0$ by taking the limit for $k\to 0$. This gives
\begin{align}
\gamma_{0,\theta}(t):= 
\begin{cases}
x_1(t,0,\theta) =\displaystyle t\sin(\theta),\\
x_2(t,0,\theta)=\displaystyle  t\cos(\theta),\\
\displaystyle x_3(t,0,\theta)=0.
\end{cases}
\end{align}
 Thus we obtain the arcs connecting the origin with points lying in $\{(x_1,x_2,x_{3}) \in \mathbb{H}\, |\, x_{3}=0\}$. Next we define the map
\begin{align}\label{formelphi}
\Phi(t,k,\theta):= \big(x_1(t,k,\theta),\, x_2(t,k,\theta),\, x_3(t,k,\theta)\big),
\end{align}
for $t\in[0, \frac{2\pi}{{|k|}}]$ , $\theta \in [0,2\pi)$, $k \in \mathbb{R}$.
The determinant of $\Phi$ is given by 
\begin{align}\label{detlemma}
 \det \Phi(t,k,\theta)= \frac{kt\sin(kt)-2(1-\cos(kt))}{k^{4}}\, ,
\end{align}
see \cite[S.161]{montii}.

\smallskip

\noindent For the proof of the following Proposition we refer to {\normalfont\cite{montii}},{\normalfont\cite{nagel}} and  {\normalfont\cite{monti}}.

\begin{prop}\label{prop} The following statements hold true. 
\begin{enumerate}
\item[a)]{Any two points in $\mathbb{H}$ can be connected by a
(not necessarily unique) geodesic.}

\smallskip

 \item[b)]{The C-C metric is invariant under left translation respectively the group law on $\mathbb{H}$, which means
 \begin{align}
 d(x,y)=d(z\boxplus x,z\boxplus y)
\end{align}
for $x,y,z\in \mathbb{H}$.}

\smallskip

 \item[c)]{The mapping 
 \begin{align}\label{hu}
 \qquad \qquad \Phi &: \menge{(t,k,\theta)\in \mathbb{R}^{3} }{ \theta \in \mathbb{R} \slash2\pi\mathbb{Z}, \ k \in \mathbb{R}, \ t \in\left(0, \frac{2\pi}{|k|} \right) } \to \mathbb{H}\setminus \PP,
\end{align}
where $\Phi$ is given in \eqref{formelphi}, is a homeomorphism and $\PP:=\{(x_{1},x_2,x_3)\in \mathbb{H} \ | \ x_1=0, x_2=0\}$.}

\smallskip

 \item[d)]{For a fixed compact set $K \subset \mathbb{H}$ there exists a constant $M>0$ such that for all $x,y \in K$ holds
\begin{align}\label{continuous}
M \norm{x-y}_{e}\leq d(x,y) \leq M^{-1}\norm{x-y}_{e}^{1/2}\, .
\end{align}
}

\smallskip

\item[e)]{Define the  dilation $r(x):=(rx_{1},rx_{2},r^{2}x_{3})$ for $x\in \mathbb{H}$ and $r>0$. Then 
\begin{align}\label{scale}
r^{4}\, B_{1}(0)= B_{r}(0).
\end{align}}
\end{enumerate}
\end{prop}

\smallskip

\noindent 
To conclude this brief overview of the C-C metric we prove the continuity of $d(\cdot)$ with respect to the Euclidean metric. 

\begin{lemma}\label{lemmadx}
Let $\Omega$ be an open bounded domain in $\mathbb{H}$. The function $d(\cdot)$ is continuous with respect to the Euclidean distance on $\mathbb{H}$.
\end{lemma}
\begin{proof}{We have to show that
\begin{align}
 |d(x)-d(y)| \leq d(x,y)
\end{align}
holds for $x,y\in \mathbb{H}$. Once the above inequality is established, the continuity of $d(\cdot)$ with respect to the Euclidean distance will follow by \eqref{continuous}. We recall that we set $d(x)=0$ for $x\in \Omega^c$. Let $x\neq y$. The case $x,y \in \Omega^{c}$ is trivial. For the case $x\in \Omega^{c}$ and $y\in \Omega$, we know that $d(x)=0$. Let use denote by $\phi(t)$ the arc of a geodesic curve connecting $x$ and $y$, which exists in view of Proposition \ref{prop}. This curve is continuous and must intersect the boundary of $\Omega$. Therefore exists $b\in \mathrm{Dom}(\phi)$ such that $\phi(b)\in \partial \Omega$. This gives
\begin{align}
d(x,y) \geq d(\phi(b),y) \geq d(y) = |d(y)-d(x)|.
\end{align}
It remains to prove the claim in the case  $x,y \in {\Omega}$. Without loss of generality we assume that $d(x)>d(y)$. Because $d$ is continuous, see  \eqref{continuous} and $\partial \Omega$ compact, there exists a $z\in \partial \Omega$ such that $d(y)=d(z,y)$. Thus we get
\begin{align}
|d(x)-d(y)|= d(x)-d(y) \leq d(x,z)-d(y,z) \leq d(x,y).
\end{align}
The last inequality follows by the triangle inequality.
}\end{proof}

\noindent After these prerequisites we can state the main result of this section.

\begin{thm}\label{mainnn}
Let $\Omega \subset \mathbb{H}$ be an open bounded domain.  The inequality 
\begin{align} \label{eq-schlauch}
 |\Omega^{\beta}| \geq |B_{ \mathrm{R}(\Omega)}(0)\setminus B_{ \mathrm{R}(\Omega)-\beta}(0)|=\left(\mathrm{R}(\Omega)^{4}-(\mathrm{R}(\Omega)-\beta)^{4}\right) |B_{1}(0)|.  
\end{align}
holds for all $\beta\in (0, \mathrm{R}(\Omega))$. 
Equality in \eqref{eq-schlauch} is achieved if $\Omega=B_r(0)$ with any $r>0$.
\end{thm}

\begin{proof}[Proof of Theorem {\normalfont\ref{mainnn}}]
{ First of all let us fix the parameter $\beta$ where $0<\beta\leq \mathrm{R}(\Omega)$.
Because $\Omega$ is open bounded and $d(\cdot)$ is continuous on $\mathbb{H}$, see Lemma \ref{lemmadx}, 
there exists an $x \in \Omega$ such that $B_{\mathrm{R}(\Omega)}(x) \subseteq \Omega$. We know that the Lebesgue measure and $d(\cdot,\cdot)$ are left-invariant respectively the group law of the Heisenberg group. Hence we can translate $\Omega$ in such a way that $x$ is the origin of its translated copy. This means that 
\begin{align}
B_{\mathrm{R}(\Omega)}(0) \subseteq \Omega.
\end{align}
Now we fix a constant ${c}>0$ such that
\begin{align}
\Omega \subseteq B_{{c}}(0),
\end{align}
which is possible in view of \eqref{continuous} because $\Omega$ is bounded. Let us begin with the lower bound on $|\Omega^{\beta}|$. From  \eqref{geod} we know that $x\in B_{2c}(0)$ if and only if 
there exist $k\in[-\pi/c ,\pi/c]$, $\theta\in [0,2\pi)$ and $0\leq t < 2c$ such that
\begin{align}\label{zq}
x=
\left(\frac{\cos(\theta)- \cos(kt+\theta)}{k}, 
\frac{ \sin(kt+\theta)-\sin(\theta)}{k}, 
\frac{tk-\sin(kt)}{2k^{2}}\right)= \Phi(t,k,\theta)\, .
\end{align}
We want to describe the points lying in $\Omega^{\beta}$ under the coordinate transformation \eqref{zq}. To this end we define the set
\begin{align}
\Omega(\Phi):=\menge{x\in \Omega^{\beta}}{\exists\,  (t,k,\theta)\in E \ \text{such that} \ x=\Phi(t,k,\theta)},
\end{align}
where 
$$
E:= (0,2c) \times  \left(-\frac{\pi}{c}, \frac{\pi}{c} \right )\times [0,2\pi). 
$$
Note that the map $\Phi: E \to \mathbb{H}$ is injective, see Proposition \ref{prop}. That means that  $\Omega^{\beta} \supseteq\Omega(\Phi)$. 
For a fixed $k\in(-\pi/c ,\pi/c)$ and $\theta\in [0,2\pi)$ we define the curve 
\begin{align}\label{qwe}
\varphi(t):=
\left(\frac{\cos(\theta)- \cos(kt+\theta)}{k}, 
\frac{ \sin(kt+\theta)-\sin(\theta)}{k}, 
\frac{tk-\sin(kt)}{2k^{2}}\right),
\end{align}
where $t\in[0,2c]$. This curve satisfies the condition $d(\varphi(t),\varphi(0))=t$ for $t\in[0,2c)$, because these are the geodesic arcs described in \eqref{geod}.  Now let
\begin{align}\label{zuhu}
a:= \inf\menge{t>0}{\varphi(t)\notin \Omega},
\end{align}
which is  well-defined since  $\varphi(0)=0\in   B_{\mathrm{R}(\Omega)}(0) \subseteq \Omega$ and  $\Omega$ is bounded. It follows that $\varphi(a)\in \partial \Omega$.  We can thus use the definition of $a$ and the inclusions $B_{\mathrm{R}(\Omega)}\subseteq\Omega \subseteq B_{c}(0)$ to obtain
\begin{align}\label{opq}
 \mathrm{R}(\Omega)\leq a \leq c.
\end{align}
Now we consider the restriction of the curve $\varphi$ on the interval $[a-\beta,a]$. Notice that this curve connects the point $\varphi(a-\beta)$ with $\varphi(a)\in \partial \Omega$. Moreover, in view of the definition of $a$, this curve is still a horizontal curve lying in $\Omega$. Therefore by the definition of the C-C metric the following estimate holds
\begin{align}
d(\varphi(t),\varphi(a)) \leq a-t < \beta \qquad \forall\ t\in(a-\beta,a).
\end{align}
From $\varphi(a)\in \partial \Omega$ we obtain $d(\varphi(t))\leq d(\varphi(t),\varphi(a))< \beta$ for all $t\in (a-\beta,a)$, which means that $\varphi(t)\in \Omega^{\beta}$ for $t\in (a-\beta,a)$. It follows that
\begin{align}\label{hzhw}
\Omega^{\beta} \supseteq \Omega(\Phi) \supseteq \menge{\Phi(t,k,\theta)\in \mathbb{H}}{(t,k,\theta)\in (a-\beta,a) \times  (-{\pi}/c, {\pi}/c) \times [0,2\pi)   ]}.
\end{align}
This inclusion and the formula \eqref{detlemma} imply
\begin{align}
|\Omega(\Phi)| \geq \int_{0}^{2\pi} \int_{-\frac \pi c}^{\frac \pi c}\int_{a-\beta}^{a} \frac{2-kt\sin(kt)-2\cos(kt)}{k^{4}} \df t \df k \df \theta.
\end{align}
 Replacing the variables $2ck$ with ${k}$ and $t/2c$ with $t$ further yields
\begin{align}\label{poio}
|\Omega^{\beta}| \geq |\Omega(\Phi)| \geq  (2c)^{4}\int_{0}^{2\pi} \int_{-2\pi}^{2\pi}\int_{\frac{a-\beta}{2c}}^{\frac{a}{2c}} \frac{2-kt\sin(kt)-2\cos(kt)}{k^{4}}\df  {t} \df  {k} \df \theta.
\end{align}
In order to obtain a suitable lower bound on the integral on the right hand-side of \eqref{poio} we notice that the function
\begin{equation} \label{ef}
f(\tau) = 2-\tau\sin(\tau)-2\cos(\tau)
\end{equation}
is non-decreasing on $[0,\pi]$. Indeed, this follows from the fact that $f(0)=f'(0)=0$ and $f''(\tau)\geq 0$ on $[0,\pi]$. 
Since the condition
\begin{equation} \label{kt-condition}
(t,k)\in \left(\frac{a-\beta}{2c} , \frac{a}{2c} \right)\times (-2\pi,2\pi)
\end{equation} 
together with the estimate in \eqref{opq} yield
\begin{align}\label{monotonicityyy}
 |tk| \ \leq\ \pi\, ,
\end{align}
it follows that, for any fixed $k \in (-2\pi, 2\pi)$ the function $2-kt\sin(kt)-2\cos(kt)$ is non-decreasing in $t$ on $(\frac{a-\beta}{2c} , \frac{a}{2c})$. 
Inequality \eqref{opq} then yields the lower bound
 \begin{align}\label{qwre}
\int_{0}^{2\pi} \int_{-2\pi}^{2\pi}\int_{\frac{a-\beta}{2c}}^{\frac{a}{2c}}\, |\det\Phi(t,k,\theta)| \df  {t}\! \df  {k}\! \df \theta\  \geq\  \int_{0}^{2\pi} \int_{-2\pi}^{2\pi}\int_{\frac{\mathrm{R}(\Omega)-\beta}{2c}}^{\frac{\mathrm{R}(\Omega)}{2c}} |\det\Phi({t},{k},\theta)| \df  {t}\! \df  {k}\! \df \theta.
\end{align}
We use this inequality to estimate the right hand-side in \eqref{poio}. This gives 
\begin{align}\label{iqzquw}
 |\Omega^{\beta}| \geq  (2c)^{4}\int_{0}^{2\pi} \int_{-2\pi}^{2\pi} \int_{\frac{\mathrm{R}(\Omega)-\beta}{2c}}^{\frac{\mathrm{R}(\Omega)}{2c}} |\det\Phi({t},{k},\theta)| \df  {t} \df  {k} \df \theta.
\end{align}
Moreover, since  $B_{\mathrm{R}(\Omega)}(0)\subseteq B_{c}(0)$, we have $\mathrm{R}(\Omega)/2c\leq 1$. From the formula of the geodesics, see equation \eqref{geod}, we then deduce that
 \begin{align}
\int_{0}^{2\pi} \int_{-2\pi}^{2\pi} \int_{\frac{\mathrm{R}(\Omega)-\beta}{2c}}^{\frac{\mathrm{R}(\Omega)}{2c}}\,  |\det\Phi({t},{k},\theta)| \df  {t} \df  {k} \df \theta= |B_{{\mathrm{R}(\Omega)/2c}}(0)|- |B_{{(\mathrm{R}(\Omega)-\beta)/2c}}(0)|.
\end{align}
This in combination with inequality \eqref{iqzquw} and the scaling properties of balls with respect to the C-C metric described in Proposition \ref{prop} thus yield the desired estimate
\begin{align}
 |\Omega^{\beta}| \geq (\mathrm{R}(\Omega)^{4}-(\mathrm{R}(\Omega)-\beta)^{4})\, |B_{1}(0)|.
\end{align}
It remains to prove the sharpness of the lower bound above. To this end we fix an $r>0$ and consider the set
$$
B_r(0)^{\beta} =  \menge{x\in B_r(0)}{d(x) < \beta}, \qquad 0< \beta <r.
$$
In order to prove that \eqref{eq-schlauch} turns into an equality for $\Omega=B_r(0)$ it suffices to show that 
\begin{align} \label{suff}
|B_r(0)^{\beta} |= \left| \menge{x \in \mathbb{H}}{ r -\beta< d(x,0)<r}\right|.
\end{align}
Inequality \eqref{eq-schlauch}, which we have already proven, shows that
$$
|B_r(0)^{\beta} \geq |\menge{x \in \mathbb{H}}{ r -\beta< d(x,0)<r}|,
$$ 
To prove the opposite inequality let $x \in  B_r(0)^{\beta}$.  Then $d(x,0)<r$ and $d(x)< \beta$.  We know that there exists $y\in \partial B_r(0)$ such that $d(x,y)=d(x)$, because $d(\cdot,\cdot)$ is continuous and $\overline{\Omega}$ is compact. Thus  by an application of the triangle inequality we get
\begin{align}
r= d(y,0) \leq d(y,x)+d(x,0) = d(x) +d(x,0) \leq \beta +d(x,0) .
\end{align}
This implies 
$$
B_r(0)^{\beta} \subseteq \menge{x \in \mathbb{H}}{ r -\beta< d(x,0)<r},
$$
and therefore completes the proof.
}
\end{proof}

\begin{remark}
Theorem {\normalfont\ref{mainnn}} can also be proven for the Euclidean metric in $\mathbb{R}^{n}$, considering the Euclidean in-radius. In that case the proof is much easier, because the determinant of the spherical coordinates is obviously monotonically increasing in the radial part and does not depend on an angle like in the case on $\mathbb{H}$.
\end{remark}

\begin{corollary}\label{corstrip}
Let $\Omega \subset \mathbb{H}$ be an open bounded domain.  For any $0<\beta\leq \mathrm{R}{(\Omega)}$ we have
 \begin{align*}
 |\Omega^{\beta}| \ \geq \ \beta\  \mathrm{R}(\Omega)^3\,  |B_1(0)|.
\end{align*}
\end{corollary}

\begin{proof}
Since 
\begin{align}
\frac{\mathrm{R}(\Omega)^{4}-(\mathrm{R}(\Omega)-\beta)^{4}}{\beta}
\end{align}
is a non-increasing function of $\beta$ on $(0,\mathrm{R}(\Omega))$, the result follows immediately from inequality \eqref{eq-schlauch}.
\end{proof}

\section{\bf Proof of the main result}
\label{sec-proof}

\subsection{Spectral decomposition}
In order to find a representation of the spectral decomposition of the Heisenberg Laplacian we introduce the Fourier transform in the $x_{3}$-direction;
\begin{align}
\mathcal{F}_3\, u(x',\xi_{3}) := \frac{1}{\sqrt{2\pi}} \int_\mathbb{R}\mathrm{e}^{-\mathrm{i}x_{3}\xi_{3} }\, u(x',x_{3})  \df x_{3},
\end{align}
where $x'=(x_1,x_2)$ and $x:=(x',x_{3})\in \mathbb{H}$. Then
\begin{align}
\fourier \, \mathrm{A}(\Omega)\,  \fourier^{*} \ = \ \left(\mathrm{i} \partial_{x_{1}}-\frac{x_2}{2}\, \xi_{3} \right)^{2}+\left(\mathrm{i}\partial_{x_2}+\frac{x_1}{2}\, \xi_{3} \right)^{2} \ = \ \left(\mathrm{i} \nabla_{x'}+\xi_{3}\magnet(x') \right)^{2},
\end{align}
where $\magnet(x'):=\frac{1}{2}(-x_{2},x_{1})$. Hence for each fixed $\xi_3\in\R$ the right hand-side is the Landau Hamiltonian in $L^2(\R^2)$ associated with the constant magnetic field $\xi_3$. Its eigenvalues are given by the Landau levels  $\{\mathrm{|\xi_{3}|(2k-1)}\}_{k\in\N}$. We denote by $ \mathrm{P_{\mathrm{k},\xi_3}}$ the orthogonal projection in $L^{2}(\mathbb{R}^{2})$ onto the Landau level  $\mathrm{|\xi_{3}|(2k-1)}$ and recall the following well-known properties:
\begin{align}\label{magnet}
\begin{split}
 \mathrm{P_{k,\xi_3}}(y,y) \ &= \ \frac{1}{2\pi}|\xi_3|, \qquad \forall\ y\in \mathbb{R}^{2}, \\
 \int_{\mathbb{R}^{2}} \left(\int_{\Omega} \left| \mathrm{P_{k,\xi_3}}(x,y) \right|^{2} \df y  \right) \df x \ &= \
  \int_{\Omega} \left(\int_{\mathbb{R}^{2}}   \mathrm{P_{k,\xi_3}}(x,y) {\mathrm{P_{k,\xi_3}}(y,x)} \df x  \right) \df y \\  
\  &=  \  \int_{\Omega} \   {\mathrm{P_{k,\xi_3}}(y,y)} \df y = \frac{|\xi_3|}{2 \pi} |\Omega|.
\end{split}
\end{align}
Hence for any $u$ such that $\fourier\, u(\cdot, \xi_{3})$ belongs to the domain of $\left(\mathrm{i} \nabla_{x'}+\xi_{3}\magnet(x') \right)^{2}$ we have
\begin{align}\label{decomposition}
\fourier \, \mathrm{A}(\Omega)\, u(x',\xi_{3}) = \sum_{\mathrm{k=1}}^{\infty} |\xi_{3}|(2\mathrm{k}-1) \int_{\mathbb{R}^{2}} \mathrm{P}_{\mathrm{k},\xi_{3}}(x',y') \fourier\, u(y',\xi_{3}) \df y'\, .
\end{align}

\begin{proof}[\bf Proof of Theorem {\normalfont \ref{dreieins}}]
We split the proof into three steps. In the first one we derive the sharp leading term with an additional negative term. The appearing negative term will be treated in the second part of the proof. The last part of the proof is dedicated to the proof of an auxiliary result needed in step two. 

\subsection{The sharp leading term}
\label{sec-leading}
In the sequel we will decompose a vector $x\in \mathbb{H}$ as 
\begin{equation} \label{x-decomp} 
x=(x',x_{3})=(x_1,x_2,x_3). 
\end{equation}
We extend the eigenfunctions  $v_{j}(x)$ of $\mathrm{A}(\Omega)$ by zero to  $x \in \Omega^{c}$ and write
\begin{align*}
\trace{\mathrm{A}(\Omega)-\lambda}_{-} =&   \sum_{j: \lam<\lambda} \lambda  \norm{v_{j}}_{L^{2}(\Omega)}^{2} - \norm{\left(\partial_{x_{1}}+\frac{1}{2}x_{2}\partial_{x_{3}}\right)v_{j}}_{L^{2}(\mathbb{R}^{3})}^{2}  - \norm{\left(\partial_{x_{2}}-\frac{1}{2}x_{1}\partial_{x_{3}}\right)v_{j}}_{L^{2}(\mathbb{R}^{3})}^{2}   \\
= & \int_{\mathbb{R}}   \sum_{j: \lam<\lambda} \lambda \left(  \norm{ \fourier\, v_{j}(\cdot,\xi_{3})}_{L^{2}(\mathbb{R}^{2})}^{2} - \norm{\left(\mathrm{i}\partial_{x_{1}}-\frac{1}{2}x_{2} \xi_{3} \right) \fourier\, v_{j}(\cdot,\xi_{3})}_{L^{2}(\mathbb{R}^{2})}^{2}\right)\!\! \df \xi_{3} \\  
&- \ \int_{\mathbb{R}}\sum_{j: \lam<\lambda}   \norm{\left(\mathrm{i}\partial_{x_{2}}   +\frac 12 x_{1}\xi_{3}\right) \fourier\, v_{j}(\cdot,\xi_{3})}_{L^{2}(\mathbb{R}^{2})}^{2} \df \xi_{3}.
\end{align*}
We apply the spectral decomposition in \eqref{decomposition} and use Fatou's lemma to obtain the following estimate for the trace:
\begin{align}\label{qiuwzwieueoeo}
\trace{\mathrm{A}(\Omega)-\lambda}_{-} \ \leq \ \int_{\mathbb{R}}   \sum_{j: \lam<\lambda}\,  \sum_{\mathrm{k}=1}^{\infty}\ (\lambda-|\xi_{3}|(2\mathrm{k}-1)) \norm{f_{j,\mathrm{k},\xi_{3}}}_{L^{2}(\mathbb{R}^{2})}^{2} 
 \df \xi_{3},
\end{align}
where
\begin{align*}
f_{{j,\mathrm{k},\xi_{3}}}(x') \ :=& \ \int_{\mathbb{R}^{2}} \mathrm{ P_{k, \xi_{3}}} (x',y') \fourier v_{j}(y', \xi_{3}) \df y' \\
=& \ \frac{1}{\sqrt{2\pi}}  \int_{\Omega} \mathrm{P_{k, \xi_{3}}} (x',y') \, \exp{-\mathrm{i} y_{3} \xi_{3}}v_{j}(y', y_{3}) \df y'  \df y_{3} \\
= &  \frac{1}{\sqrt{2\pi}}  \sprod{\mathrm{P_{k, \xi_{3}}}(x',\cdot) \, \exp{-\mathrm{i}  \xi_3 \cdot}}{ v_{j}(\cdot)}_{L^{2}(\Omega)}.
\end{align*}
Next we estimate the right hand-side of \eqref{qiuwzwieueoeo} further by considering the positive part of $(\lambda-|\xi_{3}|(2\mathrm{k}-1))$. This gives
\begin{align}\label{poiqiqpow}
\begin{split}
\trace{\mathrm{A}(\Omega)-\lambda }_{-} \  \leq & \ \int_{\mathbb{R}} \sum_{k=1}^\infty (\lambda-|\xi_{3}|(2\mathrm{k}-1) )_{+}  \sum_{j=1}^{\infty} \norm{f_{j,\mathrm{k},\xi_{3}}}_{L^{2}(\mathbb{R}^{2})}^{2}  \df \xi_{3}   \\
& -\int_{\mathbb{R}} \sum_{k=1}^\infty (\lambda-|\xi_{3}|(2\mathrm{k}-1) )_{+}\sum_{j: \lam\geq\lambda}^{\infty} \norm{f_{j,\mathrm{k},\xi_{3}}}_{L^{2}(\mathbb{R}^{2})}^{2}  \df \xi_{3}. 
\end{split}
\end{align}
Since the sequence $\{v_j\}_{j\in\N}$ is an orthonormal basis in $L^{2}(\Omega)$ we can use Parseval's identity to evaluate the sum over $j$. Taking into account  \eqref{magnet} we obtain
\begin{align} \label{vierzwei}
\sum_{j=1}^{\infty} \norm{f_{j,\mathrm{k},\xi_{3}}}_{L^{2}(\mathbb{R}^{2})}^{2}  =& \ \frac{1}{{2\pi}} \int_{\mathbb{R}^{2}} \sum_{j=1}^{\infty}  \left|  \sprod{\mathrm{P_{k, \xi_{3}}}  (x',\cdot) \, \exp{-\mathrm{i}  \xi_3 \cdot}}{ v_{j}(\cdot)}_{L^{2}(\Omega)}\right|^{2} \df x' = \  \frac{|\xi_{3}|}{4\pi^{2}}|\Omega |.
\end{align}
This allows us to calculate the first term on the right-hand side of \eqref{poiqiqpow}. We have
\begin{align*}
 \int_{\mathbb{R}} \sum_{k=1}^{\infty}(\lambda-|\xi_{3}|(2\mathrm{k}-1) )_{+} \sum_{j=1}^{\infty} \norm{f_{j,\mathrm{k},\xi_{3}}}_{L^{2}(\mathbb{R}^{2})}^{2}  \df \xi_{3} 
& = \ \frac{ |\Omega| }{2 \pi^{2}}\, \sum_{k=1}^\infty \int_{0}^\infty (\lambda-\xi_{3}(2\mathrm{k}-1))_+\, \xi_{3}   \df \xi_{3} \\
&  = \ \frac{ |\Omega| \, \lambda^3 }{12 \pi^{2}}\, \sum_{k=1}^\infty \frac{1}{(2k-1)^{2}} =  \ \frac{ |\Omega| }{96}\, \lambda^{3}\, ,
\end{align*}
where we have used the identity
\begin{align}\label{limitofseries}
 \sum\limits_{k=1}^{\infty}\frac{1}{(2k-1)^{2}}= \frac{\pi^{2}}{8}.
\end{align}
Putting together the above estimates and using \eqref{poiqiqpow} we get 
\begin{align}\label{vierdrei}
\trace{\mathrm{A}(\Omega)-\lambda }_{-} \ \leq\  \frac{ |\Omega| }{96} \,  \lambda^{3}
-\int_{\mathbb{R}} \sum_{k=1}^{\infty}( \lambda-|\xi_{3}|(2\mathrm{k}-1) )_{+} \sum_{j: \lam\geq\lambda} \norm{f_{j,\mathrm{k},\xi_{3}}}_{L^{2}(\mathbb{R}^{2})}^{2}  \df \xi_{3}.  
\end{align}
On the right hand-side we thus have the sharp leading term and an additional negative term. The latter will be treated in the next step.

\subsection{The negative lower order term}
The next step is to establish a suitable lower bound on
\begin{align} \label{Q-def}
\mathrm{Q}(\lambda,k,\xi_3)  := \sum_{j: \lam\geq\lambda} \norm{f_{j,\mathrm{k},\xi_{3}}}_{L^{2}(\mathbb{R}^{2})}^{2}.
\end{align}
Using equation \eqref{vierzwei} we rewrite the series as follows;
\begin{align}
& \mathrm{Q}(\lambda,k,\xi_3) = \ \frac{|\xi_{3}|}{4\pi^{2}}\ |\Omega|-  \sum_{j: \lam<\lambda}\norm{f_{j,\mathrm{k},\xi_{3}}}_{L^{2}(\mathbb{R}^{2})}^{2}  
\label{Q-eq} \\
   &=\frac{1}{2\pi} \int_{\mathbb{R}^{2}}  \int_{\Omega} \Big|\, \mathrm{P_{k, \xi_{3}}} (x',y')\, \exp{-\mathrm{i} y_{3}\xi_{3}}- \sum_{j: \lam<\lambda}   \sprod{\mathrm{P_{k, \xi_{3}}}  (x',\cdot) \, \exp{-\mathrm{i}  \xi_3 \cdot}}{ v_{j}(\cdot)}_{L^{2}(\Omega)} v_{j}(y)  \Big |^{2}  \df y \df x'. \nonumber
\end{align}
To estimate the right hand-side form below we consider the set 
\begin{align}
E^{\beta}:=\menge{\Phi(t,k,\theta)\in \mathbb{H}}{(t,k,\theta)\in (a-\beta,a) \times  (-{\pi}/c, {\pi}/c) \times [0,2\pi)]}.
\end{align}
Note that in view of \eqref{hzhw} we have 
$$
\Omega \supseteq {\Omega}^{\beta} \supseteq E^{\beta}\, . 
$$
By applying the inequality 
\begin{equation} \label{eq-zw}
|z-w|^{2}\ \geq \ \frac{1}{2}|z|^{2}-|w|^{2}, \qquad  z,w \in \mathbb{C}\, ,
\end{equation}
and using equation \eqref{magnet} we thus obtain
\begin{align}
\mathrm{Q}(\lambda,k,\xi_3)  &\geq  \frac{|\xi_{3}|}{8\pi^{2}} \Big|\, E^\beta\Big |- \frac{1}{2\pi} \int_{\mathbb{R}^{2}}  \int_{{E}^{\beta}}  \Big | \sum_{j: \lam<\lambda}   \sprod{\mathrm{P_{k, \xi_{3}}}  (x',\cdot)\, \exp{-\mathrm{i}  \xi_3 \cdot}}{ v_{j}(\cdot)}_{L^{2}(\Omega)} v_{j}(y)  \Big |^{2} \df y \df x' .
\end{align}
In the end of the proof of Theorem \ref{mainnn} we have shown that $|E^{\beta}|\geq |B_{\mathrm{R}(\Omega)}(0)\setminus  B_{\mathrm{R}(\Omega)-\beta}(0)|$. Moreover,  mimicking the proof of Corollary \ref{corstrip} yields
\begin{align}
 |B_{\mathrm{R}(\Omega)}(0)\setminus  B_{\mathrm{R}(\Omega)-\beta}(0)|\ \geq\ \beta\ \frac{|B_{\mathrm{R}(\Omega)}(0)|}{\mathrm{R}(\Omega)}\, .
\end{align}
Hence
\begin{align*}
\mathrm{Q}(\lambda,k,\xi_3) \ &\geq \ \beta\ \frac{|\xi_{3}|}{8\pi^{2}} \frac{|B_{\mathrm{R}(\Omega)}(0)|}{\mathrm{R}(\Omega)} \\
& - \frac{1}{2\pi} \int_{\mathbb{R}^{2}} \int_{{E}^{\beta}}   \Big | \sum_{j: \lam<\lambda}   \sprod{\mathrm{P_{k, \xi_{3}}}  (x',\cdot)\,  \exp{-\mathrm{i}\cdot \xi_{3}}}{ v_{j}(\cdot)}_{L^{2}(\Omega)} v_{j}(y)   \Big|^{2} \df y \df x' .
\end{align*}
At this point we have to estimate the negative integral from above. Note that the linear combination of $v_{j}$ lies in $d[a]$. Therefore we can use the inequality
\begin{align}\label{inequalityproof}
\int_{E^{\beta}} |u|^{2} \df x \ \leq \ \beta^{2} \int_{\Omega} |\nabla_{\mathbb{H}}\, u|^{2}  \df x, \qquad u\in d[a],
\end{align}
which is proved in section \ref{sec-aux}. Assuming for the moment that \eqref{inequalityproof} holds true we get
\begin{align}
& \frac{1}{2\pi} \int_{\mathbb{R}^{2}}  \int_{{E}^{\beta}}  \Big |\sum_{j: \lam<\lambda}  \sprod{\mathrm{P_{k, \xi_{3}}}  (x',\cdot)\,  \exp{-\mathrm{i}  \xi_3 \cdot}}{ v_{j}(\cdot)}_{L^{2}(\Omega)} v_{j}(y)  \Big |^{2} \df y  \df x' \\
& \quad  \leq   \frac{\beta^{2}}{2 \pi} \int_{\mathbb{R}^{2}}  \int_{\Omega}  \Big |\sum_{j: \lam<\lambda}  \sprod{\mathrm{P_{k, \xi_{3}}}  (x',\cdot)\,  \exp{-\mathrm{i}  \xi_3 \cdot}}{v_{j}(\cdot)}_{L^{2}(\Omega)}  \nabla_{\mathbb{H}}\, v_{j}(y)  \Big |^{2} \df y  \df x'.
\end{align}
Integration by parts and the fact that the eigenfunctions  $v_{j}$ are mutually orthogonal then yield
\begin{align*}
& \frac{1}{2\pi} \int_{\mathbb{R}^{2}}  \int_{{E}^{\beta}} \Big |\sum_{j: \lam<\lambda}  \sprod{\mathrm{P_{k, \xi_{3}}}  (x',\cdot)\, \exp{-\mathrm{i}  \xi_3 \cdot}}{ v_{j}(\cdot)}_{L^{2}(\Omega)} v_{j}(y) \Big |^{2} \df y  \df x' \\
& \quad \leq
\frac{\beta^{2}}{2 \pi} \int_{\mathbb{R}^{2}}  \int_{\Omega} \sum_{j: \lam<\lambda}  \lambda_j(\Omega)    \Big | \sprod{\mathrm{P_{k, \xi_{3}}}  (x',\cdot) \, \exp{-\mathrm{i}  \xi_3 \cdot}}{v_{j}(\cdot)}_{L^{2}(\Omega)} v_{j}(y) \Big |^{2}  \df y  \df x'  \\
& \quad \leq \frac{\beta^{2}\, \lambda}{2 \pi}  \int_{\mathbb{R}^{2}}   \sum_{j: \lam<\lambda}    \Big| \sprod{\mathrm{P_{k, \xi_{3}}}  (x',\cdot)\,  \exp{-\mathrm{i}  \xi_3 \cdot}}{v_{j}(\cdot)}_{L^{2}(\Omega)}  \Big |^{2}  \df x' \, .
\end{align*}
Finally we sum over all $j$ and use \eqref{vierzwei} to obtain
\begin{align}
 \int_{\mathbb{R}^{2}}  \sum_{j: \lam<\lambda}  \left| \sprod{\mathrm{P_{k, \xi_{3}}}  (x',\cdot)\, \exp{-\mathrm{i}  \xi_3 \cdot}}{v_{j}(\cdot)}_{L^{2}(\Omega)} \right|^{2}  \df x' \ \leq \ \frac{|\xi_3||\Omega|}{2\pi}\, .
\end{align}
Summarizing these estimates we arrive at the following lower bound on $\mathrm{Q}$:
\begin{align} \label{aux-lb}
 \mathrm{Q}(\lambda,k,\xi_3) \ \geq& \ \beta \frac{|\xi_{3}|}{8\pi^{2}} \frac{|B_{\mathrm{R}(\Omega)}(0)|}{\mathrm{R}(\Omega)}-   \beta^{2}  \frac{|\xi_{3}|}{4 \pi^{2}} |\Omega| \lambda= \beta \frac{|\xi_3|}{8 \pi^{2}} \left(  \frac{|B_{\mathrm{R}(\Omega)}(0)|}{\mathrm{R}(\Omega)}-   2\beta   |\Omega| \lambda \right).
\end{align}
Now we  set
\begin{equation} \label{beta-choice}
\beta \ := \ \frac{|B_{\mathrm{R}(\Omega)}(0)|}{4 |\Omega|\, \mathrm{R}(\Omega)}\  \lambda^{-1}.
\end{equation} 
We have to show that with this choice $\beta \leq \mathrm{R}(\Omega)$ holds true. By \eqref{hl-li-yau} 
\begin{align}
 \frac{1}{\lambda_{1}(\Omega)}\leq \frac{3}{8\sqrt{2}}\,  |\Omega|^{1/2} \leq  |\Omega|^{1/2}.
\end{align}
This inequality in combination with $|B_{\mathrm{R}(\Omega)}(0)|\leq |\Omega|$ yields that for any  $\lambda\geq \lambda_{1}(\Omega)$ we have
\begin{align}
 \beta= \frac{|B_{\mathrm{R}(\Omega)}(0)|}{4 |\Omega|\, \mathrm{R}(\Omega) \lambda} \leq \frac{|B_{\mathrm{R}(\Omega)}(0)|}{|\Omega|\, \mathrm{R}(\Omega) \lambda_{1}(\Omega)} \leq  \frac{|B_{\mathrm{R}(\Omega)}(0)|}{\sqrt{|\Omega|}\ \mathrm{R}(\Omega)}  \leq \frac{\sqrt{|B_{\mathrm{R}(\Omega)}(0)|}}{\mathrm{R}(\Omega)}.
\end{align}
From Proposition \ref{prop}(e) and  the fact that $|B_{1}(0)|\leq 1$, see e.g.~\cite{sachovka}, we thus deduce that 
\begin{align}
\beta \, \leq\, \frac{\sqrt{|B_{\mathrm{R}(\Omega)}(0)|}}{\mathrm{R}(\Omega)}= \mathrm{R}(\Omega)\sqrt{|B_{1}(0)|} \leq \mathrm{R}(\Omega).
\end{align}
as required. Hence we may insert \eqref{beta-choice} into \eqref{aux-lb}, which yields
\begin{align}\label{nebenterm}
 \mathrm{Q}(\lambda,k,\xi_3) \ \geq& \ \frac{|\xi_{3}|}{64 \pi^{2}} \frac{|B_{\mathrm{R}(\Omega)}(0)|^{2}}{\mathrm{R}(\Omega)^{2}|\Omega|}\  \lambda^{-1}.
\end{align}
Finally we estimate the sum of the negative integral of \eqref{vierdrei}
\begin{align}
\trace{\mathrm{A}(\Omega)-\lambda }_{-}\leq \frac{ |\Omega|}{96}\, \lambda^{3}
-  \frac{1}{64\pi^{2}} \frac{|B_{\mathrm{R}(\Omega)}(0)|^{2}}{\mathrm{R}(\Omega)^{2}|\Omega|}\, \lambda^{-1}\int_{\mathbb{R}} \sum_{k=1}^{\infty}( \lambda-|\xi_{3}|(2\mathrm{k}-1) )_{+} \,  |\xi_{3}|  \df \xi_{3}.  
\end{align}
and calculate the integral on the right hand-side by using the substitution $\xi_{3}(2k-1)=s$ and  \eqref{limitofseries}:
\begin{align*}
2 \sum_{\mathrm{k}=1}^{\infty} \int_{0}^{\infty} (\lambda-\xi_{3}(2\mathrm{k}-1))_{+} {\xi_{3}} \df \xi_{3} \ = \ \sum_{\mathrm{k}=1}^{\infty} \frac{2}{(2\mathrm{k}-1)^{2}} \int_{0}^{\infty} s(\lambda-s)_{+}   \df s \ =  \frac{\pi^2 \lambda^3}{24}.
\end{align*}
This together with Proposition \ref{prop}(e) yields inequality \eqref{hauptgleichung}. It thus remains to prove  \eqref{inequalityproof}. 

\subsection{Proof of inequality \eqref{inequalityproof}} 
\label{sec-aux}
Without loss of generality we can assume that $u\in C_{0}^{\infty}(\Omega)$. Note that $\Omega \subset B_{c}(0)$. 
Hence in the coordinates given by the parametrization of the ball $B_{2c}(0)$ and with the help of \eqref{detlemma} we obtain
\begin{align}
\int_{E^{\beta}} |u(x)|^{2} \df x= \int_{0}^{2\pi}\int_{-\frac{\pi}{c}}^{\frac{\pi}{c}}\int_{a-\beta}^{a} |u(t,k,\theta)|^{2} \,  \frac{f(tk)}{k^{4}} \df \theta,
\end{align}
where $f$ is defined in \eqref{ef}. Wa can assume again that $k$ is positive. Otherwise we substitute $k$ by $-k$ and use that $f(\cdot)$ is even. We know that $u(a,k,\theta)=0$ for all $k\in(-\pi/c,\pi/c)$ and $\theta\in[0,2\pi)$. Since $f(\cdot)$ is increasing on $[0,\pi]$ and $|tk|\leq \pi$, see $\eqref{monotonicityyy}$, it easily follows that 
 \begin{align}
 \sup\limits_{a-\beta \leq \tau\leq a} \int_{a-\beta}^{\tau}{f(sk)}\df s  \int_{\tau}^{a} \, \frac{1}{f(sk)} \df s  \ \leq \  \sup\limits_{a-\beta \leq \tau \leq a} (\tau-a+\beta)(a-\tau) =  \frac{\beta^{2}}{4}\, .
\end{align}
In view of \cite[Theorem 1.14]{kufner} we thus conclude with
\begin{align}
 \int_{0}^{2\pi}\int_{-\frac{\pi}{c}}^{\frac{\pi}{c}}\int_{a-\beta}^{a} |u(t,k,\theta)|^{2} \,  \frac{f(tk)}{k^{4}}\df t\! \df k\! \df \theta\, \leq\,  \frac{\beta^{2}}{4} \int_{0}^{2\pi}\int_{-\frac{\pi}{c}}^{\frac{\pi}{c}}\int_{a-\beta}^{a} |\partial_{t}u(t,k,\theta)|^{2} \,  \frac{f(tk)}{k^{4}} \df t\! \df k\! \df \theta.
\end{align}
Moreover, we know that $E^{\beta}\subseteq B_{2c}(0)$. Hence
\begin{align}
  \int_{0}^{2\pi}\int_{-\pi/c}^{\pi/c}\int_{a-\beta}^{a} |\partial_{t}u(t,k,\theta)|^{2}\   \frac{f(tk)}{k^{4}} \df t\! \df k\! \df \theta \leq  \int_{0}^{2\pi}\int_{-\frac{\pi}{c}}^{\frac{\pi}{c}}\int_{0}^{2c} |\partial_{t}u(t,k,\theta)|^{2}\,  \frac{f(tk)}{k^{4}} \df t\! \df k\! \df \theta.
\end{align}
Let us now turn to the coordinate system $(x_1,x_2,x_3)$. Keeping in mind the parametrization \eqref{geod} we get
\begin{align}
\int_{B_{2c}(0)}   \left| \partial_{t}u \right|^{2} \df x = \int_{B_{2c}(0)}   \left| \partial_{x_1}u\, \partial_{t}x_1+ \partial_{x_2} u\, \partial_{t} x_2+ \partial_{x_3}u\, \partial_{t} x_3\right|^{2} \df x.
\end{align}
From the differential equation of the geodesics; $2\partial_{t} x_3(t)=x_2(t) \partial_{t}x_1(t)-\partial_{t}x_2(t)x_1(t)$, it further follows that 
\begin{align}
\int_{B_{2c}(0)}   \left| \partial_{t}u \right|^{2} \df x  = \int_{B_{2c}(0)}   \left|   \partial_{t} x_1 X_1 u+  \partial_t x_2 X_2\, u \right|^{2} \df x.
\end{align}
The cross terms will be estimated with the help of the Cauchy-Schwarz inequality and $2ab\leq a^{2}+b^{2},\ a,b,\in \mathbb{R}$. This gives 
\begin{align}
2\sprod{ \partial_{t}x_1\, X_1 u}{ \partial_{t}x_2 \, X_2 u} \leq \norm{ \partial_{t}x_2 X_1 u}_{L^{2}(B_{2c}(0))}^{2} + \norm{ \partial_{t} x_1 X_2 u }_{L^{2}(B_{2c}(0))}^{2}.
\end{align}
Now we collect all the above estimates  and use the fact that the support of the function $u$ lies in $\Omega$ to arrive at 
\begin{align}
\int_{E^{\beta}}  {|u|^{2}} \df x \leq \beta^{2} \left(\norm{ \partial_{t} x_1 X_1 u}_{L^{2}(\Omega)}^{2} + \norm{ \partial_{t} x_2 X_1u}_{L^{2}(\Omega)}^{2}+\norm{ \partial_{t}x_2 X_2 u }_{L^{2}(\Omega)}^{2} + \norm{ \partial_{t}x_1 X_2 u }_{L^{2}(\Omega)}^{2} \right)
\end{align}
From \eqref{geod} we see that $ \partial_{t}x_1= \sin(kt+\theta)$ and  $ \partial_{t} x_2= \cos(kt+\theta)$, which implies inequality \eqref{inequalityproof} completing thus the proof of Theorem \ref{dreieins}.
\end{proof}

\section{\bf{Improved spectral estimates}}\label{sec-improvement}
We have seen in Theorem \ref{dreieins}, that for a bounded domain $\Omega \subset \mathbb{H}$ we improved the sharp bound for the eigenvalue sum by adding a negative term of the form $-\lambda^{2}\mathcal{C}(\Omega)$, where $\mathcal{C}(\Omega)$ is a positive constant only depending on the geometry of $\Omega$.  The order of $\lambda$  can be improved if we assume the validity of a Hardy inequality with respect to the C-C metric. In particular we introduce

\begin{assumption}\label{Hardyassumption}
Let $\Omega \subset \mathbb{H}$ be a bounded domain. We assume that there exists a constant $c\in[2,\infty)$ independent of $\Omega$ such that
\begin{align} \label{hardy-cc}
\int_{\Omega} \frac{|u(x)|^{2}}{d(x)^{2}} \df x \ \leq \ c^{2} \int_{\Omega} \left( |X_{1}\, u(x)|^{2}+|X_{2}\, u(x)|^{2} \right) \df x
\end{align}
holds for all $u\in C_0^\infty (\Omega)$. 
\end{assumption}
Note that the sharp value of the constant is $c=2$.  Therefor consider the sequence  $g_{\varepsilon}=d(x)^{1/2+\varepsilon}$ and using the Eikonal equation, see \cite[Thm. 3.1]{montiii}, i.e.
\begin{align}
|X_{1} \, d(x)|^{2}+|X_{2}\, d(x)|^{2}=1 \qquad \text{a.e.} \quad x\in\Omega.
\end{align}
It remains to check, that $d$ lies in the domain of the quadratic form in \eqref{einseins}. Since $X_1d(x)$ and $X_2d(x)$ exist almost everywhere on $\Omega$, see \cite{montiii}, an additional application of the Eikonal equation yields that $d$ is weakly differentiable respectively $X_1$ and $X_2$. At that point it can be shown by standard convolution arguments that $d$ can be approximated by $C_0^\infty(\Omega)$ functions. Then, by a direct calculation we obtain
$$
\lim_{\eps\to 0} \frac{\int_{\Omega} \left( |X_{1} \, g_\eps|^{2}+|X_{2}\, g_\eps|^{2} \right) \df x}{\int_{\Omega} \frac{|g_\eps|^{2}}{d(x)^{2}} \df x} \, =\, 4\, ,
$$
which shows that $c=2$ is the best value of the constant in \eqref{hardy-cc}.

\begin{remark}
If $\Omega$ is a $C^{1,1}$ regular boundary, then Assumption \ref{Hardyassumption} holds true, although the constant is unknown, see \cite[Thm. 4.1 and p.120]{danielli}.
\end{remark}

\begin{thm}\label{improvedestimate}
Let $\Omega \subset \mathbb{H}$ be a bounded domain and let 
\begin{align*}
\sigma(\Omega) := \inf\limits_{0<\beta \leq \RR(\Omega)} \frac{|\Omega^\beta|}{\beta}\, .
\end{align*}
Under Assumption {\normalfont\ref{Hardyassumption}} we have 
\begin{align}\label{improveig}
\trace{\mathrm{A}(\Omega)-\lambda }_{-} \ \leq\ \max\left\{0, \frac{ |\Omega| }{96} \,  \lambda^{3}
-  \frac{1+2/c}{96}\sigma(\Omega)^{\frac{2c+2}{c+2}} (4c+4)^{\frac{-2c-2}{c+2}}|\Omega|^{\frac{-1}{1+2/c}} \lambda^{2+\frac{1}{c+2}}  \right\}.
\end{align}
\end{thm}

\noindent Note that the quantity $\sigma(\Omega)$ is strictly positive because of {Corollary} \ref{corstrip}. For the proof of Theorem \ref{improvedestimate}, we need the following Lemma:
\begin{lemma}\label{lemmahardy}Under Assumption {\normalfont\ref{Hardyassumption}} it holds
\begin{align}
\int_{\Omega^{\beta}} |u(x)|^{2} \df x \leq c^{2+2/c} \beta^{2+2/c} \norm{\AM(\Omega)u}_{L^{2}(\Omega)} \norm{\AM(\Omega)^{1/c}u}_{L^{2}(\Omega)} 
\end{align}
for all $u\in \mathrm{Dom}(A(\Omega))$, where $\Omega^{\beta}:= \{x\in \Omega\, |\,  d(x)<\beta \}$.
\end{lemma}
\begin{proof} 
Since the Eikonal equation still holds for $d$, see \cite{montiii}, and $d$ lies in the domain of the quadratic dorm, which was discussed in that section, the claim is proved in the same way as \cite[Thm.~4, p.169]{davissss}.
\end{proof}

\begin{proof}[Proof of Theorem {\normalfont\ref{improvedestimate}}]
Since $\Omega$ satisfies the assumptions of Theorem \ref{dreieins}, we can follow the proof of the latter. From section \ref{sec-leading}, in particular from equation
\eqref{vierdrei}, we infer that 
\begin{align}\label{vierdrei-bbbb}
\trace{\mathrm{A}(\Omega)-\lambda }_{-} \ \leq\  \frac{ |\Omega| }{96} \,  \lambda^{3}
-\int_{\mathbb{R}} \, \sum_{k=1}^{\infty}\ ( \lambda-|\xi_{3}|(2\mathrm{k}-1) )_{+} \,  \mathrm{Q}(\lambda,k,\xi_3) \df \xi_{3}.  
\end{align}
with $\mathrm{Q}(\lambda,k,\xi_3)$ given by \eqref{Q-def}. For $\beta\in (0, \mathrm{R}(\Omega))$ we consider the set
\begin{align}
\Omega^{\beta}:= \{x\in \Omega\, |\,  d(x)<\beta \}.
\end{align}
Obviously it holds $\Omega^{\beta}\subseteq \Omega$. Hence equations \eqref{Q-eq} and \eqref{eq-zw} imply that 
\begin{align} \label{QQQ-below}
\mathrm{Q}(\lambda,k,\xi_3)  &\geq  \frac{|\xi_{3}|\, |\Omega^\beta|}{8\pi^{2}} - \frac{1}{2\pi} \int_{\mathbb{R}^{2}}  \int_{\Omega^\beta}  \Big | \!\sum_{j: \lam<\lambda}   \sprod{\mathrm{P_{k, \xi_{3}}}  (x',\cdot)\, \exp{-\mathrm{i}  \xi_3 \cdot}}{ v_{j}(\cdot)}_{L^{2}(\Omega)} v_{j}(y)  \Big |^{2} \df y \df x'.
\end{align}
An application of Lemma \ref{lemmahardy} yields
\begin{equation}\label{justificationofbetainconvex}
\begin{split}
&  \int_{\mathbb{R}^{2}} \int_{\Omega^\beta}  \Big | \sum_{j: \lam<\lambda}   \sprod{\mathrm{P_{k, \xi_{3}}}  (x',\cdot)\, \exp{-\mathrm{i}  \xi_3 \cdot}}{ v_{j}(\cdot)}_{L^{2}(\Omega)} v_{j}(y)  \Big |^{2}\! \df y \df x' \\ 
\qquad& \qquad \leq c^{2+2/c}  \beta^{2+2/c} \lambda^{1+1/c} \int_{\mathbb{R}^{2}} \sum_{j: \lam<\lambda}   \Big | \sprod{\mathrm{P_{k, \xi_{3}}}  (x',\cdot)\, \exp{-\mathrm{i}  \xi_3 \cdot}}{ v_{j}(\cdot)}_{L^{2}(\Omega)}\Big |^2 \df x'\\
& \qquad \qquad \qquad \leq c^{2+2/c}  \beta^{2+2/c} \lambda^{1+1/c}   \frac{|\xi_{3}|}{2\pi}|\Omega|.
\end{split}
\end{equation}
For the last inequality we used \eqref{vierzwei}. Inserting this result into \eqref{QQQ-below} and using the definition of $ |\Omega^\beta|$ we find that
\begin{align*}
\mathrm{Q}(\lambda,k,\xi_3)  &\geq  \frac{|\xi_{3}|\, \sigma(\Omega) }{8\pi^{2}} \beta - c^{2+2/c}  \beta^{2+2/c} \lambda^{1+1/c}   \frac{|\xi_{3}|}{4\pi^2}|\Omega|\\
&= \frac{|\xi_{3}|}{8\pi^{2}} \beta \left(\sigma(\Omega) - 2c^{2+2/c}  \beta^{1+2/c} \lambda^{1+1/c} |\Omega| \right)
\end{align*}
holds true uniformly in $k$, where
\begin{align*}
\sigma(\Omega) = \inf\limits_{0<\beta \leq \RR(\Omega)} \frac{|\Omega^\beta|}{\beta}.
\end{align*}
From {Corollary} \ref{corstrip} we know that $\sigma(\Omega)>0$. Hence upon setting
\begin{equation} \label{betagiven-2}
\beta := \sigma(\Omega)^{\frac{1}{1+2/c}} (4+4/c)^{\frac{-1}{1+2/c}} c^{-1-\frac{1}{1+2/c}}  |\Omega|^{\frac{-1}{1+2/c}} \lambda^{\frac{-1-1/c}{1+2/c}}
\end{equation}
we obtain 
\begin{equation} \label{Q-lowerb-2222}
\mathrm{Q}(\lambda,k,\xi_3)  
\ \geq \ \frac{|\xi_{3}|}{4\pi^{2}}{(1+2/c)}\sigma(\Omega)^{\frac{2c+2}{c+2}} (4c+4)^{\frac{-2c-2}{c+2}}|\Omega|^{\frac{-1}{1+2/c}} \lambda^{-1+\frac{1}{c+2}}  \qquad \forall\ k\in\N. 
\end{equation}
However, as in the proof of Theorem \ref{dreieins} we have to verify that $\beta$ given by \eqref{betagiven-2} with $\lambda\geq \lambda_1(\Omega)$ satisfies 
$\beta\leq \RR(\omega)$. An application of Lemma \ref{lemmahardy} for $u=v_{1}$ and $\beta=\RR(\Omega)$ yields $1\leq \RR(\Omega)^{2+2/c}c^{2+2/c} \lambda_{1}(\Omega)^{1+1/c}$. The latter inequality shows that 
\begin{align*}
 \beta^{1+2/c} \ &\leq \ \sigma(\Omega) (4+4/c)^{-1} c^{-2-2/c}  |\Omega|^{-1} \lambda_{1}(\Omega)^{-1-1/c} \\
 & \leq \ \sigma(\Omega) (4+4/c)^{-1} \RR(\Omega)^{2+2/c} |\Omega|^{-1}\\
  & \leq \ (4+4/c)^{-1} \RR(\Omega)^{1+2/c} \ \leq \ \RR(\Omega)^{1+2/c}.
\end{align*}
holds for all $\lambda\geq  \lambda_1(\Omega)^{-1}$ and therefore justifies the choice of $\beta$ in  \eqref{betagiven-2}. 
Inequality \eqref{improveig} now follows by inserting the lower bound \eqref{Q-lowerb-2222} in \eqref{vierdrei-bbbb} and evaluating the integral in $\xi_3$ and then the series in $k$ as in the proof of Theorem \ref{dreieins}.
\end{proof}
\begin{example}
Let us set $\Omega=B_{r}(0):= \{x\in \mathbb{H}\, |\, d(x,0)<r\}$ for $r>0$. In \cite{yang} it was shown, that for $B_{r}(0)$  Assumption \ref{Hardyassumption} holds true with the constant $c=2$. From Theorem \ref{improvedestimate} and the lower estimate for $\sigma(\Omega)$, which is given by Corollary \ref{corstrip}, we obtain
\begin{align}
\trace{\mathrm{A}(\Omega)-\lambda }_{-} \ \leq\ \max\left\{0, \frac{ |\Omega| }{96} \,  \lambda^{3}
-  \frac{1}{2^{7}3^2\sqrt{3}}\, r^{-\frac{3}{2}}\,  |\Omega| \lambda^{2+\frac{1}{4}}  \right\}.
\end{align}
\end{example}

\section*{\bf Aknowledgements}
 H.~K. was supported by the Gruppo Nazionale per Analisi Matematica, la Probabilit\`a e le loro Applicazioni (GNAMPA) of the Istituto Nazionale di Alta Matematica (INdAM).
The support of MIUR-PRIN2010-11 grant for the project  ``Calcolo delle variazioni'' (H.~K.) is also gratefully acknowledged. 

B.~R. was supported by the German Science Foundation through the Research Training Group 1838: \textit{Spectral Theory and Dynamics of Quantum Systems.}

\bibliographystyle{plain}

\bibliography{heisenberg}


\end{document}